\documentclass[reqno]{amsart}
\usepackage{amsmath,amsfonts,amssymb,amscd}
\usepackage{stmaryrd}
\usepackage[arrow,matrix,cmtip,curve]{xy}
\usepackage{fullpage}
\usepackage{tikz}
\usetikzlibrary{positioning,arrows}

\numberwithin{equation}{section}

\theoremstyle{definition}
\newtheorem{theorem}{Theorem}[section]

\newtheorem{corollary}[theorem]{Corollary} 
\newtheorem{definition}[theorem]{Definition} 
\newtheorem{example}[theorem]{Example}
\newtheorem{hypothesis}[theorem]{Hypothesis} 
\newtheorem{lemma}[theorem]{Lemma}

\newtheorem{thm}{Theorem} 

\DeclareMathOperator\Aut{Aut}
\DeclareMathOperator\diag{diag}
\DeclareMathOperator\Ext{Ext}

\DeclareMathOperator\GL{GL}

\DeclareMathOperator\GrMod{GrMod}
\DeclareMathOperator\hdet{hdet}
\DeclareMathOperator\Hom{Hom}

\newcommand\grp[1]{\langle #1 \rangle}
\newcommand\inv{^{-1}}
\newcommand\iso{\cong}
\newcommand\kk{\Bbbk}
\newcommand\op{\mathrm{op}}
\newcommand\tensor{\otimes}
\newcommand\tornado{\xi}

\newcommand\NN{\mathbb N}

\newcommand\degs{s}
\newcommand\dimd{d}
\newcommand\source{\mathfrak{s}}
\newcommand\target{\mathfrak{t}}

\newenvironment{psmatrix}
  {\left(\begin{smallmatrix}}
  {\end{smallmatrix}\right)}

\usepackage{xcolor}

\begin{document}

\title{Four-vertex quivers supporting twisted graded Calabi--Yau algebras}

\author[Gaddis]{Jason Gaddis}
\address{(Gaddis, Nguyen, Wright) Department of Mathematics, Miami University, 301 S. Patterson Ave., Oxford, Ohio 45056} 
\email{gaddisj@miamioh.edu, nguye123@miamioh.edu, wrigh245@miamioh.edu}

\author[Lamkin]{Thomas Lamkin}
\address{(Lamkin) Department of Mathematics, University of California San Diego, La Jolla, California 92903}
\email{tlamkin@ucsd.edu}

\author[Nguyen]{Thy Nguyen}

\author[Wright]{Caleb Wright}

\begin{abstract}
We study quivers supporting twisted graded Calabi--Yau algebras. 
%building on work of Rogalski and the first author. 
Specifically, we classify quivers on four vertices in which the Nakayama automorphism acts on the degree zero part by either a four-cycle, a three-cycle, or two two-cycles. In order to realize algebras associated to some of these quivers, we show that graded twists of a twisted graded Calabi--Yau algebra is another algebra of the same type.
\end{abstract}

\maketitle

\section{Introduction}

Calabi--Yau algebras are pervasive in the mathematics and mathematical physics literature. In the philosophy of Ginzburg \cite{ginz}, $\NN$-graded Calabi--Yau algebras arising in nature should have relations determined by a superpotential.
While that analysis focused on the connected case, that is, deformations of the polynomial ring, the theory has grown in multiple ways since that seminal work. 
%We now may consider Calabi--Yau algebras that appear as a quotient of the path algebra of a quiver. 
By a result of Bocklandt, a graded Calabi--Yau algebra of (global) dimension three is the path algebra of a quiver with relations coming from a homogeneous superpotential \cite[Theorem 3.1]{Bock}. 

One can also consider \emph{twisted} graded Calabi--Yau algebras, where the potential is twisted by some automorphism (the \emph{Nakayama} automorphism). We now have a reasonably clear understanding of twisted graded Calabi--Yau algebras in global dimension two and three due to the work of Reyes and Rogalski \cite{RR1,RR2}. Under suitable hypotheses, these algebras appear as, or are Morita equivalent to, the path algebra of a quiver with relations coming from a twisted superpotential.
What we do not understand, in general, is which quivers may appear. This article is a contribution to further understanding such quivers.

For single vertex quivers, either the quiver has two loops with a degree four superpotential, or it has three loops and a degree three superpotential \cite{ASch}. In \cite{GR1}, the authors
%Rogalski and the first author 
classified two- and three-vertex quivers supporting twisted graded Calabi--Yau algebras of global and Gelfand--Kirillov dimension three. 

We consider quivers on four vertices and, as in \cite{GR1}, 
%This paper is a continuation of that work wherein we consider quivers on four vertices. As in \cite{GR1}, 
we split our analysis according to the permutation matrix associated to the Nakayama automorphism. In this paper, we only consider certain permutation matrices and hope to complete the analysis in a later paper. We refer to the next section for terminology related to the following theorem.

\begin{thm}[Theorems \ref{thm.4cycle}, \ref{thm.3cycle}, \ref{thm.22cycle}]\label{thm.main}
Let $Q$ be a quiver with $|Q_0|=4$, let $M$ denote the adjacency matrix of $Q$, and let $\omega$ be a homogeneous superpotential on $Q$ with $\degs=\deg(\omega)$. Suppose $A=\kk Q/(\partial_a(\omega) : a \in Q_1)$ is a twisted graded Calabi--Yau algebra of dimension 3 and Gelfand--Kirillov dimension 3. If $P$ is the restriction of the Nakayama automorphism of $A$ to $A_0$, where $P$ corresponds to a four-cycle, a three-cycle, or two two-cycles, then $M$ is one of the following:
\[
\renewcommand{\arraystretch}{1.5}
\begin{array}{|c|c|l|} \hline
\degs & P & \text{possible adjacency matrices $M$} \\ \hline
3 &
\begin{psmatrix}0 & 1 & 0 & 0 \\ 0 & 0 & 1 & 0 \\ 0 & 0 & 0 & 1 \\ 1 & 0 & 0 & 0\end{psmatrix}
& 
\begin{psmatrix}0 & 0 & 0 & 3 \\ 3 & 0 & 0 & 0 \\ 0 & 3 & 0 & 0 \\ 0 & 0 & 3 & 0\end{psmatrix},
\begin{psmatrix}0 & 1 & 0 & 2 \\ 2 & 0 & 1 & 0 \\ 0 & 2 & 0 & 1 \\ 1 & 0 & 0 & 2\end{psmatrix},
\begin{psmatrix}0 & 1 & 2 & 0 \\ 0 & 0 & 1 & 2 \\ 2 & 0 & 0 & 1 \\ 1 & 2 & 0 & 0\end{psmatrix},
\begin{psmatrix}1 & 1 & 1 & 0 \\ 0 & 1 & 1 & 1 \\ 1 & 0 & 1 & 1 \\ 1 & 1 & 0 & 1\end{psmatrix},
\begin{psmatrix}1 & 1 & 0 & 1 \\ 1 & 1 & 1 & 0 \\ 0 & 1 & 1 & 1 \\ 1 & 0 & 1 & 1\end{psmatrix},
\begin{psmatrix}2 & 1 & 0 & 0 \\ 0 & 2 & 1 & 0 \\ 0 & 0 & 2 & 1 \\ 1 & 0 & 0 & 2\end{psmatrix},
\begin{psmatrix}0 & 1 & 1 & 1 \\ 1 & 0 & 1 & 1 \\ 1 & 1 & 0 & 1 \\ 1 & 1 & 1 & 0\end{psmatrix}^* \\
3 &
\begin{psmatrix}0 & 0 & 1 & 0 \\ 1 & 0 & 0 & 0 \\ 0 & 1 & 0 & 0 \\ 0 & 0 & 0 & 1\end{psmatrix}
&
\begin{psmatrix}0 & 0 & 0 & 1 \\ 0 & 0 & 0 & 1 \\ 0 & 0 & 0 & 1 \\ 1 & 1 & 1 & 2\end{psmatrix},
\begin{psmatrix}1 & 0 & 1 & 1 \\ 1 & 1 & 0 & 1 \\ 0 & 1 & 1 & 1 \\ 1 & 1 & 1 & 0\end{psmatrix},
\begin{psmatrix}0 & 1 & 1 & 1 \\ 1 & 0 & 1 & 1 \\ 1 & 1 & 0 & 1 \\ 1 & 1 & 1 & 0\end{psmatrix}^* \\
3 &
\begin{psmatrix}0 & 1 & 0 & 0 \\ 1 & 0 & 0 & 0 \\ 0 & 0 & 0 & 1 \\ 0 & 0 & 1 & 0\end{psmatrix}
&
\begin{psmatrix}1 & 0 & 2 & 0 \\ 0 & 1 & 0 & 2 \\ 0 & 2 & 1 & 0 \\ 2 & 0 & 0 & 1\end{psmatrix},
\begin{psmatrix}0 & 1 & 1 & 1 \\ 1 & 0 & 1 & 1 \\ 1 & 1 & 0 & 1 \\ 1 & 1 & 1 & 0\end{psmatrix},
\begin{psmatrix}1 & 1 & 1 & 0 \\ 1 & 1 & 0 & 1 \\ 1 & 0 & 1 & 1 \\ 0 & 1 & 1 & 1\end{psmatrix},
\begin{psmatrix}0 & 2 & 1 & 0 \\ 2 & 0 & 0 & 1 \\ 1 & 0 & 0 & 2 \\ 0 & 1 & 2 & 0\end{psmatrix},
\begin{psmatrix}1 & 0 & 1 & 1 \\ 0 & 1 & 1 & 1 \\ 1 & 1 & 0 & 1 \\ 1 & 1 & 1 & 0\end{psmatrix},
\begin{psmatrix}1 & 1 & 1 & 0 \\ 1 & 1 & 0 & 1 \\ 1 & 0 & 0 & 2 \\ 0 & 1 & 2 & 0\end{psmatrix},
\begin{psmatrix}1 & 1 & 1 & 0 \\ 1 & 1 & 0 & 1 \\ 0 & 1 & 1 & 1 \\ 1 & 0 & 1 & 1\end{psmatrix}^* \\ 
4 &
\begin{psmatrix}0 & 1 & 0 & 0 \\ 1 & 0 & 0 & 0 \\ 0 & 0 & 0 & 1 \\ 0 & 0 & 1 & 0\end{psmatrix}
&
\begin{psmatrix}1 & 0 & 1 & 0 \\ 0 & 1 & 0 & 1 \\ 0 & 1 & 1 & 0 \\ 1 & 0 & 0 & 1\end{psmatrix},
\begin{psmatrix}0 & 0 & 1 & 1 \\ 0 & 0 & 1 & 1 \\ 1 & 1 & 0 & 0 \\ 1 & 1 & 0 & 0\end{psmatrix},
\begin{psmatrix}1 & 0 & 1 & 0 \\ 0 & 1 & 0 & 1 \\ 0 & 1 & 0 & 1 \\ 1 & 0 & 1 & 0\end{psmatrix}^*,
\begin{psmatrix}0 & 0 & 2 & 0 \\ 0 & 0 & 0 & 2 \\ 1 & 1 & 0 & 0 \\ 1 & 1 & 0 & 0\end{psmatrix}^* \\
\hline
\end{array}
\]
Moreover, each matrix $M$ appears as the adjacency matrix of a twisted graded Calabi--Yau algebra satisfying the above, except possibly those indicated by $^*$.
\end{thm}

The matrices in Theorem \ref{thm.main} appear largely in standard ways. In most cases, we construct a skew group ring $\kk[x,y,z]\#\kk G$ for some finite group $G$ acting by graded automorphisms on $\kk[x,y,z]$. In this paper, $G$ is typically the cyclic group of order four or it is the Klein-4 group. In some cases, it is necessary to replace $\kk[x,y,z]$ by a three-dimensional Artin--Schelter regular algebra.

Other matrices we are able to realize through Ore extensions of skew group rings $\kk[x,y]\#\kk G$. In \cite{GR1}, all matrices could be constructed either through skew group rings or Ore extensions except for one family of \emph{mutations} of McKay matrices. In this paper we observe that these techniques are not sufficient to realize all such matrices. It is necessary to also consider graded (Zhang) twists of twisted graded Calabi--Yau algebras,
%We will also consider graded (Zhang) twists of graded Calabi--Yau algebras, 
which are again twisted graded Calabi--Yau as a consequence of the following result.

\begin{thm}[Theorem \ref{thm.AStwist}]
\label{thm.intro_twist}
If $A$ is generalized Artin--Schelter regular of dimension $\dimd$, then any graded twist of $A$ is also generalized Artin--Schelter regular of dimension $\dimd$.
\end{thm}

This paper is organized as follows. In Section \ref{sec.background} we provide necessary background on derivation-quotient algebras and the twisted Calabi--Yau condition. In Section \ref{sec.examples} we demonstrate common ways to produce twisted Calabi--Yau algebras, focusing on methods necessary to produce examples who adjacency matrices appear in Theorem \ref{thm.main}. The bulk of work in this paper is in Section \ref{sec.quivers} where we compute the quivers that support twisted graded Calabi--Yau algebras on four vertices and prove Theorem \ref{thm.main}.

\subsection*{Acknowledgments}
Lamkin was partially supported by the Miami University Dean's Scholar program. The authors appreciate comments from Manuel Reyes and Daniel Rogalski that assisted in the proof of Theorem \ref{thm.intro_twist}. The authors also would like to thank the referee for suggestions that improved the exposition in the paper.

\section{Background}\label{sec.background}

Throughout, let $\kk$ denote an algebraically closed field of characteristic zero. All algebras in this paper are assumed to be $\kk$-algebras. We say an algebra $A$ is \emph{$\NN$-graded} (or just \emph{graded}) if there is a vector space decomposition $A = \bigoplus_{k \geq 0} A_k$ such that $A_k \cdot A_\ell \subset A_{k+\ell}$. {The symbol $\delta_{ij}$ denotes the Kronecker delta function.}

\subsection{Quivers and path algebras}

A \emph{quiver} $Q$ is a tuple $(Q_0,Q_1,\source,\target)$ consisting of a set of vertices $Q_0$, a set of edges $Q_1$, and functions $\source,\target:Q_1 \to Q_0$ that identify for each edge $a \in Q_1$ its \emph{source} $\source(a) \in Q_0$ and its \emph{target} $\target(a) \in Q_0$. The \emph{adjacency matrix} of $Q$, denoted $M_Q$, is the $n \times n$ matrix such that $(M_Q)_{ij}$ records the number of edges in $Q_1$ from vertex $i$ to vertex $j$.

A \emph{path} in $Q$ is a sequence $p=a_1a_2\cdots a_m$ with $a_i \in Q_1$ such that $\target(a_i)=\source(a_{i+1})$ for $i=1,\hdots,m-1$. Thus we may also extend $\source$ and $\target$ to paths by $\source(p)=\source(a_1)$ and $\target(p)=\target(a_m)$.
The quiver $Q$ is said to be \emph{connected} if it is connected as an undirected graph, and $Q$ is \emph{strongly connected} if for any two vertices $i$ and $j$, there exists a path $p$ with $\source(p)=i$ and $\target(p)=j$.

The \emph{path algebra} $\kk Q$ has $\kk$-basis consisting of paths in $Q$ and multiplication defined as concatenation, where the product $pq$ of two paths $p$ and $q$ is zero if $\target(p) \neq \source(q)$. For each vertex $v$, there is a \emph{trivial path} $e_v$ where $e_vp = p$ if $\source(p)=v$ and otherwise $e_vp=0$. Similarly, $pe_v=p$ if $\target(p)=v$ and otherwise $pe_v=0$. Note that if $Q_0$ is finite, then $1=\sum_{v \in Q_0} e_v$.

Let $Q_k$ denote the $\kk$-vector space of paths of length $k$ (along with $0$). Then $\kk Q = \bigoplus_{k \in \NN} Q_k$, so $\kk Q$ is graded where $(\kk Q)_0 \iso \kk^n$, $n=|Q_0|$. A nonzero element $q$ in $\kk Q$ is \emph{homogeneous} if $q \in Q_k$ for some $k$.

Let $\sigma$ be an automorphism of $\kk Q$. A homogeneous element $\omega\in Q_m$ is a \emph{$\sigma$-twisted superpotential} (of degree $\degs$) if it is invariant under the map $a_1a_2\cdots a_\degs \mapsto \sigma(a_\degs)a_1a_2\cdots a_{\degs-1}$. 

Given $a \in Q_1$, the derivation operator $\partial_a$ is defined on paths $p=a_1a_2\cdots a_m$ by $\partial_a(p) = a\inv p = a_2 \cdots a_m$ if $a=a_1$ and otherwise $\partial_a(p)=0$. Then $\partial_a$ extends linearly to all of $\kk Q$. In particular, given a homogeneous $\sigma$-twisted superpotential $\omega$ on $\kk Q$, the corresponding \emph{derivation-quotient algebra} is $\kk Q/(\partial_a(\omega) : a \in Q_1)$. Since $\omega$ is homogeneous, this algebra inherits the $\NN$-grading on $\kk Q$.

If $A=\kk Q/(\partial_a(\omega) : a \in Q_1)$ with $\omega$ homogeneous and $|Q_0|=n$, then we label the trivial paths in $Q$ as $e_1,\hdots,e_n$ so that $A_0 = \kk e_1 + \cdots + \kk e_n$. Note that $e_ie_j=\delta_{ij}$. We set
$(H_k)_{ij} = \dim_{\kk} e_i A_k e_j$. Then 
the \emph{(matrix valued) Hilbert series} of $A$ is the formal power series
\[ h_A(t) = \sum_{k=0}^\infty H_k t^k \in M_n(\kk)\llbracket t \rrbracket.\]

\subsection{Twisted graded Calabi--Yau algebras}

Given an algebra $A$, let $A^{\op}$ denote the opposite algebra and $A^e = A \tensor_\kk A^{\op}$ the \emph{enveloping algebra} of $A$. A $\kk$-central $(A,A)$-bimodule $M$ is a left $A^e$-module where the action is given by $(a \tensor b^{\op})\cdot m = amb$ for all $m \in M$. The algebra $A$ is said to be \emph{homologically smooth (over $\kk$)} if $A$ has a finite length resolution by finitely generated projective $A^e$-modules. 

\begin{definition}\label{defn.CY}
An algebra $A$ is said to be \emph{twisted Calabi--Yau} of dimension $\dimd$ if $A$ is homologically smooth and there exists an invertible $(A,A)$-bimodule $U$ such that there are right $A^e$-module isomorphisms $\Ext_{A^e}^i(A,A^e) \iso \delta_{i\dimd} U$.
\end{definition}

For a locally finite graded algebra $A$, we denote by $J(A)$ the graded Jacobson radical of $A$. Since $A_{\geq 1} \subset J(A)$, then $S=A/J(A)=A_0/J(A_0)$ is a finite-dimensional semisimple algebra. If, in addition, $A$ is twisted Calabi--Yau of dimension $\dimd$, then $A$ has graded global dimension $\dimd$ and there is a $\kk$-central invertible $(S,S)$-bimodule $V$ such that there are $(S,S)$-bimodule isomorphisms $\Ext_A^i(S,A) \iso \delta_{i\dimd} V$ \cite[Theorem 5.15]{RR2}. That is, $A$ is \emph{generalized Artin--Schelter regular}.

\begin{theorem}[{\cite[Theorem 6.8, Remark 6.9]{BSW}}]\label{thm.presentation}
Let $Q$ be a connected quiver and let $I$ be a homogeneous ideal in $\kk Q$.  If $A=\kk Q/I$ is twisted graded Calabi--Yau of dimension 3, then $A=\kk Q/ (\partial_a(\omega) : a \in Q_1)$ for some homogeneous $\sigma$-twisted superpotential $\omega$.
\end{theorem}

For the remainder of this paper, we will work exclusively in the setting of the following hypothesis, which is essentially the same as that used in \cite{GR1}.

\begin{hypothesis}\label{hyp.main}
Let $A \iso \kk Q/(\partial_a(\omega) : a \in Q_1)$ be a derivation-quotient algebra where $Q$ is a strongly connected quiver and $\omega$ is a $\sigma$-twisted superpotential of degree $\degs$ for some automorphism $\sigma$ of $\kk Q$ with $s=3$ or $s=4$. Assume that $A$ is twisted graded Calabi--Yau of dimension 3 and Gelfand--Kirillov dimension $3$.
\end{hypothesis}

Note that the hypothesis that $Q$ is strongly connected is actually a consequence of the other conditions \cite[Lemma 8.5]{RR1}.

Suppose $A=\kk Q/(\partial_a(\omega) : a \in Q_1)$ satisfies Hypothesis \ref{hyp.main}. 
Let $U$ be the bimodule appearing in Definition \ref{defn.CY}. By \cite[Proposition 5.2]{RR1}, $U=A$ as a $\kk$-vector space and there is a graded automorphism $\mu$ of $A$, called the \emph{Nakayama automorphism of $A$}, such that the action on $U$ is given by $b \cdot a \cdot c = ba\mu(c)$. Since $\mu$ is necessarily a graded automorphism, it permutes the vertices of $Q$. Let $P$ denote the associated permutation matrix where $P_{ij}=\delta_{\mu(i),j}$. We say that $(M,P,\degs)$ is the \emph{type} of $A$.

A twisted graded Calabi--Yau algebra is connected ($A_0=\kk$) it and only if it is Artin--Schelter regular \cite[Lemma 1.2]{RRZ}. Hence, when $Q$ has just one vertex, the only types satisfying Hypothesis \ref{hyp.main} are $(3,I,3)$ and $(2,I,4)$ \cite{ASch}. The types in which $Q$ has two or three vertices are described in \cite{GR1}. Eventually we hope to classify all the types that satisfy Hypothesis \ref{hyp.main}. In this paper we will focus on the types in which $|Q_0|=4$ and $P$ corresponds to a four-cycle, a three-cycle, or two two-cycles.

\begin{theorem}[{\cite{RR1}}]\label{thm.mpoly}
Suppose $A \iso \kk Q/(\partial_a(\omega) : a \in Q_1)$ 
satisfies Hypothesis \ref{hyp.main} and let
%is twisted graded Calabi--Yau where $Q$ is connected. Let 
$(M,P,\degs)$ be the type of $A$. Then the matrices $M$ and $P$ commute, and we have $h_A(t) = (p(t))\inv$, where
\begin{align}\label{eq.mpoly}
p(t)=I-Mt+PM^Tt^{\degs-1} - Pt^\degs.
\end{align}
Moreover, every zero of $\det p(t) \in \kk[t]$ is a root of unity and $\det p(t)$ vanishes at $t=1$ to order at least $3$.
\end{theorem}

As a consequence of Theorem \ref{thm.mpoly}, $\det(p(t))$ is a product of cyclotomic polynomials. The next lemma was used implicitly in \cite{GR1}. Here we state it more generally.

\begin{lemma}\label{lem.palindrome}
Suppose $A \iso \kk Q/(\partial_a(\omega) : a \in Q_1)$ satisfies Hypothesis \ref{hyp.main} and let $(M,P,\degs)$ be the type of $A$. Let $n=|Q_0|$ and let $p(t)$ be the associated matrix polynomial as in \eqref{eq.mpoly}. Then $\det(p(t))$ is palindromic if $(-1)^n\det(P)=1$ and antipalindromic otherwise.
\end{lemma}
\begin{proof}
Since $P$ is orthogonal,
\[ t^\degs p(t\inv) = It^\degs - Mt^{\degs-1} + PM^Tt - P= -P (I - M^Tt + P\inv Mt^{\degs-1} - P\inv t^\degs) = -P p(t)^T.\]
Consequently,
\[ t^{n\degs} \det(p(t\inv)) = \det(t^\degs p(t\inv)) = \det(-P p(t)^T) = (-1)^n \det(P) \det(p(t)),\]
and the result follows.
\end{proof}

\section{Examples of Calabi--Yau algebras}\label{sec.examples}

\subsection{McKay quivers}

Let $G$ be a finite group and $V$ a finite-dimensional representation of $G$ (over the field $\kk$). Let $W_1,\hdots,W_n$ be the irreducible representations of $G$. The \emph{McKay quiver} associated to the pair $(G,V)$ is the quiver with $n$ vertices, corresponding to the $n$ irreducible representations of $G$, and the number of arrows from vertex $i$ to vertex $j$ is the number of copies of $W_i$ appearing in the direct sum decomposition of $V \tensor W_j$. Equivalently, if $M$ is the adjacency matrix for the McKay quiver, then $M_{ij} = \dim_{\kk} \Hom_{\kk G}(W_i,V \tensor W_j)$.

Let $R=\bigoplus_{k \geq 0} R_k$ be a connected graded algebra that is generated in degree 1. Let $G$ be a group that acts on $R$ as graded automorphisms, so that $R_1$ is a $G$-representation. The skew group ring $R\#\kk G$ is, as a vector space, $R\tensor \kk G$, and multiplication is defined by
\[ (r \tensor g)(s \tensor h) = r g(s) \tensor gh\]
for all $r \tensor g, s \tensor h \in R \tensor G$, extended linearly. 

If $R$ is twisted graded Calabi--Yau of dimension $\dimd$, then $A=R\#\kk G$ is also twisted graded Calabi--Yau of dimension $\dimd$. One may choose a full idempotent $e \in A_0$ such that $B=e(R\#\kk G)e$ is a graded elementary algebra with $B_0 = \kk^m$. Moreover, $B$ is Morita equivalent to $A$, whence it is twisted graded Calabi--Yau \cite[Theorem 4.3, Lemma 6.2]{RR1}. It then follows that $B \iso \kk Q/I$ where $Q$ is the McKay quiver of the action of $G$ on $R_1$ and $I \subset \kk Q_{\geq 2}$ is a homogeneous ideal \cite[Lemma 3.4]{RR1}.
Putting this together gives that $A$ is Morita equivalent to a derivation-quotient algebra on $Q$.

The action of the Nakayama automorphism $\mu_A$ on $A_0 \iso \kk G$ is given by $\mu_A(g)=\hdet(g) g$, where $\hdet$ is the \emph{homological determinant} of the action of $G$ on $R$ \cite[Theorem 4.1]{RRZ}. In general, the homological determinant, as defined by J\"{o}rgensen and Zhang \cite{gourmet}, is difficult to compute. However, if $R$ is a (commutative) polynomial ring, then we may identity $g \in G$ with a linear transformation in $\GL_n(\kk)$ and in this case $\hdet(g)=\det(g)$. More generally, if $R$ is itself a derivation-quotient algebra with twisted superpotential $\omega$, then $g(\omega) = \hdet(g)\omega$ \cite[Theorem 3.3]{MS}. The \emph{winding automorphism} $\Xi^l_{\hdet}$ of $\kk G$ scales group elements by the homological determinant, so $\Xi^l_{\hdet}(g)=\hdet(g)g$.

One may decompose $1 \in \kk G$ as a sum of central idempotents $1=e_1 + \cdots + e_m$. Let $\chi_i$ be the character corresponding to the representation $W_i$. One has
\begin{align}\label{eq.ideps}
e_i = \frac{\dim_{\kk}(W_i)}{|G|} \sum_{g \in G} \chi_i(g\inv) g.
\end{align}
Thus, $\Xi^l_{\hdet}$ permutes the idempotents and this in turn determines a permutation of the vertices of $Q$.

\begin{example}\label{ex.mckay2}
Consider the quiver with adjacency matrix
\[ M = \begin{psmatrix}0 & 1 & 1 & 1 \\ 1 & 0 & 1 & 1 \\ 1 & 1 & 0 & 1 \\ 1 & 1 & 1 & 0\end{psmatrix}.\]
Let $G=C_2 \times C_2$, which has character table
\[
\begin{array}{c|cccc}
        & 1 & g & h & gh \\ \hline
\chi_1  & 1 & 1 & 1 & 1 \\
\chi_2  & 1 & 1 & -1& -1 \\
\chi_3  & 1 & -1& 1 & -1 \\
\chi_4  & 1 & -1&-1 & 1
\end{array}\]
and let $W_i$ be the irreducible (one-dimensional) representation with character $\chi_i$, $i=1,2,3,4$. Set $V=W_2\oplus W_3\oplus W_4$ and let $\psi$ be the character of $V$. The direct sum decomposition of $V \tensor W_j$ corresponds directly to the character decomposition of $\psi \cdot \chi_j$. Using this we see that $V \tensor W_j = \bigoplus_{i \neq j} W_i$ and this determines the columns of the McKay quiver associated to this action. It follows that the McKay quiver of $V$ has adjacency matrix $M$.

Let $R=\kk[x,y,z]$. Then the representation $V$ corresponds to the action of $G=C_2 \times C_2 = \grp{g,h}$ on $R_1=\kk\{x,y,z\}$ by
\[ g \mapsto \diag(-1,1,-1), \quad h \mapsto \diag(1,-1,-1).\]
By \eqref{eq.ideps}, the central idempotents of $R\#G$ are
%Then it is clear that the McKay quiver associated to this action is $M$ with central idempotents
\begin{align*}
e_1 &= \frac{1}{4}\left(1+g+h+gh\right), &
e_2 &= \frac{1}{4}\left(1+g-h-gh\right), \\
e_3 &= \frac{1}{4}\left(1-g+h-gh\right), &
e_4 &= \frac{1}{4}\left(1-g-h+gh\right).
\end{align*}
Since $\deg(g)=1$ and $\det(h)=1$, then the winding automorphism fixes the vertices. Thus the permutation matrix corresponding to the Nakayama automorphism on $R\#G$ is the identity.
\end{example}

\subsection{Ore extensions}

Let $R$ be an algebra, let $\rho \in \Aut_\kk(R)$, and let $\delta$ be a $\rho$-derivation of $R$, so 
\[ \delta(rr')=\rho(r)\delta(r') + \delta(r)r' \quad\text{for all $r,r' \in R$}.\] 
The \emph{Ore extension} $A=R[t;\rho,\delta]$ is the algebra generated over $R$ by $t$ with relations $tr=\rho(r)t + \delta(r)$ for all $r \in R$. If $R$ is twisted graded Calabi--Yau of dimension $\dimd$, then by \cite[Theorem 3.3]{LWW2}, $A$ is twisted graded Calabi--Yau of dimension $\dimd+1$. Moreover, the Nakayama automorphism of $A$ is related to that of $R$. We refer to \cite{GR1,LWW2} for more details and remark here only the computational tools necessary for our analysis.

Suppose $R$ is twisted graded Calabi--Yau of type $(M,P,2)$. Assume that $\rho$ is a graded automorphism of $R$ and $\delta$ is a graded $\sigma$-derivation. Hence, $\rho$ permutes the idempotents of $R$, so by an abuse of notation we write $\rho(e_i)=e_{\rho(i)}$. Let $P'$ be the matrix corresponding to the action of $\sigma$ on $R_0$, so $(P')_{ij} = \delta_{\rho(i),j}$. Now it follows that $A$ is twisted graded Calabi--Yau of type $(M+(P')\inv, P(P')\inv, 3)$.

\begin{example}\label{ex.ore}
Consider the action of $C_4=\grp{g}$ on $\kk[x,y]$ given by $\diag(1,i)$. Then $B=\kk[x,y]\# \kk C_4$ is twisted graded Calabi--Yau of dimension 2. So, $B$ has type $(M,P,2)$ where
\[  M = \begin{psmatrix}1 & 0 & 0 & 1 \\ 0 & 1 & 1 & 0  \\ 1 & 0 & 1 & 0 \\  0 & 1 & 0 & 1\end{psmatrix}, \quad
P = \begin{psmatrix}0 & 0 & 1 & 0 \\ 0 & 0 & 0 & 1 \\ 0 & 1 & 0 & 0 \\ 1 & 0 & 0 & 0\end{psmatrix}.\] 
We give two examples of Ore extensions of $B$.

(1) Let $\rho$ be the automorphism of $B$ induced by the automorphism of $\kk C_4$ determined by $g \mapsto -g$. Consequently, $\rho$ interchanges the idempotents $e_1 \leftrightarrow e_2$ and $e_3 \leftrightarrow e_4$. The matrix associated to $\rho$ is
\[P' = \begin{psmatrix}0 & 1 & 0 & 0 \\ 1 & 0 & 0 & 0 \\ 0 & 0 & 0 & 1 \\ 0 & 0 & 1 & 0\end{psmatrix}.\]
Consequently, $A=B[t;\rho]$ is twisted graded Calabi--Yau of dimension 3 and $A$ has type $(M + (P')\inv, P(P')\inv, 3)$. Explicitly, 
\[ M + (P')\inv=\begin{psmatrix}1 & 1 & 0 & 1 \\ 1 & 1 & 1 & 0 \\ 0 & 1 & 1 & 1 \\ 1 & 0 & 1 & 1\end{psmatrix}, \qquad
P(P')\inv = \begin{psmatrix}0 & 0 & 0 & 1 \\ 0 & 0 & 1 & 0 \\ 1 & 0 & 0 & 0 \\ 0 & 1 & 0 & 0\end{psmatrix}.\]

(2) Let $\rho$ be the automorphism of $B$ induced by the automorphism of $\kk C_4$ determined by $g \mapsto g^3$. Hence, $\rho$ fixes $e_1$ and $e_2$ but swaps $e_3$ and $e_4$. The matrix associated to $\sigma$ is
\[P' = \begin{psmatrix}1 & 0 & 0 & 0 \\ 0 & 1 & 0 & 0 \\ 0 & 0 & 0 & 1 \\ 0 & 0 & 1 & 0\end{psmatrix}.\]
Consequently, $A=B[t;\rho]$ is twisted graded Calabi Yau of dimension 3 and $A$ has type $(M + (P')\inv, P(P')\inv, 3)$ where
\[ M + (P')\inv=\begin{psmatrix}2 & 0 & 0 & 1 \\ 0 & 2 & 1 & 0  \\ 1 & 0 & 1 & 1 \\  0 & 1 & 1 & 1\end{psmatrix}, \qquad
P(P')\inv = \begin{psmatrix}0 & 0 & 0 & 1 \\ 0 & 0 & 1 & 0 \\ 0 & 1 & 0 & 0 \\ 1 & 0 & 0 & 0\end{psmatrix}.\]
\end{example}

\subsection{Graded (``Zhang'') twists}\label{sec.twist}

Let $A$ be an locally finite graded algebra. Let $\tau$ be a graded automorphism of $A$. The graded (left) twist of $A$, denoted $B={}_\tau A$, has the same $\kk$-basis as $A$ but a new multiplication $\star$ such that for $x \in B_m = A_m$ and $y \in B_n=A_n$,
\[ x \star y = \tau^n(x)y\]
where the multiplication on the right takes place in $A$. Such a twist appears in the work of Artin, Tate, and Van den Bergh \cite{ATV2} but was generalized and studied in greater detail in the work of Zhang \cite{ztwist}. Hence, such graded twists are often called \emph{Zhang twists}. We first establish that such twists preserve the twisted graded Calabi--Yau property in our setting, then provide an example which will be useful in the sequel.

Let $M$ be a graded (left) $A$-module. That is, $M=\bigoplus_{k \in \NN} M_k$ and $A_n M_k \subset M_{n+k}$. Now if $\tau$ is a graded automorphism of $A$ and $B={}_\tau A$ as above, then $M$ has a natural structure as a $B$-module, which we denote by ${}_\tau M$. As an abelian group, $M={}_\tau M$. Let $a \in A_n$ and $m \in M_k$, then the action is given by $a \star m = \tau^k(a) m$.  

The objects in the category $\GrMod A$ are graded $A$-modules and the maps are degree zero module homomorphisms. In \cite{ztwist}, Zhang considers the functor  $F:\GrMod A \to \GrMod B$ which, at the level of abelian groups, is just the identity. That is, $F(M)={}_\tau M$ and given a degree zero $A$-module homomorphism $\psi:M \to N$, $F(\psi) = \psi$. This follows since for $a \in A_n$ and $m \in M_k$ we have $\tau^k(a) \in B_n$ and $m \in ({}_\tau M)_k$, so
\[ \psi(a \star m) = \psi( \tau^k(a) m) = \tau^k(a) \psi(m) = a \star \psi(m).\]
By \cite[Theorem 3.1]{ztwist}, the functor $F:\GrMod A \to \GrMod B$ is an equivalence of categories. That the twisted graded Calabi--Yau property is invariant under twisting is a consequence of the following more general result. 

\begin{theorem}\label{thm.AStwist}
Let $A$ be a locally finite graded algebra which is generalized Artin--Schelter regular of dimension $\dimd$. Let $\tau$ be a graded automorphism of $A$. Then ${}_\tau A$ is generalized Artin--Schelter regular of the same global dimension.
\end{theorem}
\begin{proof}
Set $B={}_\tau A$. Let $I$ be a maximal homogeneous left ideal of $A$. It is clear by the definition of a twist that ${}_\tau I$ is a homogeneous left ideal of $B$. Moreover, since twisting preserves module inclusion, then ${}_\tau I$ is maximal in $B$. Since $A={}_{\tau\inv} B$, then twisting provides a one-to-one correspondence between maximal left ideals of $A$ and $B$. Since $J(A)$ is the intersection of all homogeneous left $A$-modules, it follows that $J(B)=J({}_\tau A) = {}_\tau J(A)$. Then if $S=A/J(A)$ we have ${}_\tau S = {}_\tau (A/J(A)) = {}_\tau A/{}_\tau J(A) = B/J(B)$. In fact, because $S=A_0/J(A_0)$, then it is clear that ${}_\tau S = S$.

The graded global dimension of $A$ is equal to the projective dimension of $S$ as a (left) $A$-module \cite[Corollary 4.14]{RR2}. Consequently, $S$ has a minimal resolution by graded projectives $P_\bullet \to S \to 0$. The functor $F:\GrMod A \to \GrMod B$ is an equivalence of graded module categories and so it preserves this resolution. That is, there is a resolution $({}_\tau P)_\bullet \to {}_\tau S \to 0$ of $B$-modules. Using the equivalence establishes that this resolution is minimal. 

Again let $P_\bullet \to S$ be a projective resolution of $S$. Then the equivalence $F$ gives an equivalence of cochain complexes
\begin{align*}
\xymatrix{
0 \ar[r]  & \Hom_A(P_0,A) \ar[r]^{d_0} \ar[d]^F & \Hom_A(P_1,A) \ar[r] \ar[d]^F & \cdots \\
0 \ar[r] & \Hom_B({}_\tau (P_0),B) \ar[r]^{F(d_0)} & \Hom_B({}_\tau (P_1),B) \ar[r] & \cdots
}
\end{align*}
To see that this is indeed an equivalence of cochain complexes, take $g \in \Hom_A(P_0,A)$. Then on one hand we have $F(d_0 \circ g) \in \Hom_B({}_\tau (P_1),B)$ and on the other hand we have $F(d_0) \circ F(g) \in \Hom_B({}_\tau (P_1),B)$. But since $F$ is just the identity on maps, then $F(d_0) \circ F(g) = F(d_0 \circ g)$.

Now taking cohomology gives an isomorphism $\Ext_A^i(S,A) \iso \Ext_B^i({}_\tau S,B)$. By the generalized Artin--Schelter Gorenstein condition, $\Ext_A^i(S,A) = 0$ for all $i \neq d$. On the other hand, $\Ext_A^d(S,A)$ has the structure of a right $A$-module induced from the $A$-module structure on $\Hom(-,A)$. In particular, if $f \in \Hom(M,A)$ then $f \cdot a$ is given by $(f \cdot a)(x)= f(x)a$ for all $x \in M$. By the Artin--Schelter Gorenstein condition, $\Ext_A^d(S,A)$ is a right $S$-module. That is, $\Ext_A^d(S,A) \cdot J(A) = 0$. But the right action is untwisted and so this holds also for $\Ext_B^i({}_\tau S,B)$. More precisely, by \cite[Theorem 5.2(d)]{RR2}, there exists an invertible graded $(S,S)$-bimodule $V$ such that $\Ext_A^d(S,A) \iso V$ as right $S$-modules. It follows that the condition holds as well for $\Ext_B^i({}_\tau S,B)$.
\end{proof}

\begin{corollary}\label{cor.twist}
Let $A$ be a twisted graded Calabi--Yau algebra which is generated in degree one. Let $\tau$ be a graded automorphism of $A$. Then ${}_\tau A$ is twisted graded Calabi--Yau algebra of the same global dimension.
\end{corollary}

Let $A$ be a twisted graded Calabi--Yau algebra as above and write $A \iso \kk Q/I$ where $I \subset \kk Q_{\geq 2}$. Let $(M,P,\degs)$ be the type of $A$. Let $\tau$ be a graded automorphism of $A$ and set $B={}_\tau A$. Hence, $B=\kk Q'/I'$ for some quiver $Q'$ and some ideal $I'$. It follows from the above that $Q_0=(Q')_0$. The number of arrows from vertex $i$ to vertex $j$ in $Q'$ is the dimension of $e_i \star B_1 \star e_j = \tau(e_i)A_1e_j$, which is the number of arrows from vertex $\tau(i)$ to vertex $j$ in $Q$. Thus we have $M_{Q'}=NM$ where $N$ is the permutation automorphism determined by $\tau$. The automorphism on $Q_0$ determines the action on $Q_1$ and so the permutation matrix on the vertices for $B$ is $NP$. It is further clear that $\tau$ preserves the degrees of the relations, so that $B$ is of type $(NM,NP,\degs)$.

\begin{example}\label{ex.twist}
Consider the action of the alternating group $A_4=\langle g,h : g^3=h^2=1, ghg=hg^2h \rangle$ on $\kk[x,y,z]$ corresponding to the (irreducible) representation
\[ g \mapsto \begin{psmatrix}1 & 0 & 0 \\ 0 & 0 & 1 \\ -1 & -1 & -1\end{psmatrix}, \quad
h \mapsto \begin{psmatrix}0 & 1 & 0 \\ 1 & 0 & 0  \\ -1 & -1 & -1\end{psmatrix}.\]
Since $\det(g)=\det(h)=1$, then the matrix $P$ corresponding to the Nakayama automorphism of $A=\kk[x,y,z]\#\kk A_4$ is the identity. The McKay matrix $M$ of this action is also the adjacency matrix of $A$ when realized as the quotient of a path algebra. In particular, we have
\[M=\begin{psmatrix}0 & 0 & 0 & 1 \\ 0 & 0 & 0 & 1 \\ 0 & 0 & 0 & 1 \\ 1 & 1 & 1 & 2\end{psmatrix}\]
and so the type of $A$ is $(M,I,3)$.

Now consider the twisting map induced by the automorphism of $\kk A_4$ given by $\tau(g)=\omega g$ and $\tau(h)=h$ where $\omega$ is a primitive third root of unity. This map preserves the one three-dimensional representation and permutes the three one-dimensional representations of $A_4$. Consequently, the matrix corresponding to the action of $\tau$ on the path algebra is 
\[ N = \begin{psmatrix}0 & 0 & 1 & 0 \\ 1 & 0 & 0 & 0 \\ 0 & 1 & 0 & 0 \\ 0 & 0 & 0 & 1\end{psmatrix}.\]
Since $MN=M$, it follows that the type of the twisted graded Calabi--Yau algebra ${}_\tau A$ is $(M,N,3)$.
\end{example}

\section{Quivers}\label{sec.quivers}

In this section we prove the various parts of Theorem \ref{thm.main}. Our analysis requires several routines written primarily for Sage\footnote{This code is available on github at code.ncalgebra.org. Everything in this paper can be run from fourvert.sage. The primary functions, however, are in the dependency file cypack.sage.},
some of which are adapted from Maple code written for \cite{GR1}. At one point, we also need some computations done in Mathematica.

Throughout this section, we assume $|Q_0|=4$. The subsections below correspond to the restriction of the Nakayama automorphism to $Q_0$. In particular, subsection \ref{sec.4cycle} considers a four-cycle, \ref{sec.3cycle} corresponds to a three-cycle, and \ref{sec.22cycle} corresponds to two two-cycles.

\subsection{Four cycle}\label{sec.4cycle}

In this section, we assume that the restriction $P$ of the Nakayama automorphism to $Q_0$ is a four-cycle. Hence, up to reordering the vertices, we may assume that
\begin{align}\label{eq.4cycle}
P = \begin{psmatrix}0 & 1 & 0 & 0 \\ 0 & 0 & 1 & 0 \\ 0 & 0 & 0 & 1 \\ 1 & 0 & 0 & 0\end{psmatrix}.
\end{align}
The analysis in this case is similar to \cite[Proposition 5.1]{GR1}, however we also make use of computational tools which simplify the argument at a certain point.

Recall that a $n \times n$ matrix $M$ is \emph{circulant} if $M_{ij}=M_{i+1,j+1}$ where indices are taken mod $n$. To simplify notation, we use $C(a,b,c,d)$ to denote the circulant ($4 \times 4$)  matrix whose first row is $(a,b,c,d)$.

\begin{lemma}\label{lem.4cycle}
Let $P$ be as in \eqref{eq.4cycle} and let $(M,P,\degs)$ be a type of an algebra satisfying Hypothesis \ref{hyp.main} with $|Q_0|=4$. Then $\degs=3$ and $M=C(a,b,c,d)$ for some nonnegative integers $a,b,c,d$. The possible tuples $(a,b,c,d)$ such that the corresponding $\det(p(t))$ has all roots on the unit circle are as follows:
\[
(0, 0, 0, 3), %(0,3,0,0)
(0, 1, 1, 1),
(0, 1, 0, 2),
(0, 1, 2, 0), %(0, 2, 0, 1),
(1, 1, 1, 0), %(1, 0, 1, 1),
(1, 1, 0, 1),
(2, 1, 0, 0).
\]
\end{lemma}
\begin{proof}
Since $M$ and $P$ commute, then $M=C(a,b,c,d)$ for some nonnegative integers $a,b,c,d$. Consequently, $M$ is normal and so the spectral radius satisfies $\rho(M)=6-\degs$ \cite[Proposition 8.10 (3)]{RR1}. As $\rho(M)$ is bounded below by the minimal row sum and bounded above by the maximal row sum \cite[Theorem 8.1.22]{HJ}, we have $a+b+c+d=6-\degs$.

Since $M$ and $P$ are both normal and they commute, then they are simultaneously diagonalizable. One could at this point utilize a diagonalization argument similar to \cite[Proposition 5.1]{GR1}. Here we will rely on computational tools\footnote{Run four\_chk(s).}. First, we form a list of all tuples $(a,b,c,d)$ satisfying $a+b+c+d=6-\degs$. Then we construct the corresponding circulant matrix and compute its matrix polynomial as in \eqref{eq.mpoly}. This is then compared to the list of products of cyclotomic matrix polynomials of degree $4\degs$. The result is a list of matrices whose matrix polynomial has roots all lying on the unit circle. This computation shows that there are no such matrices when $\degs=4$.
\end{proof}

It is left to show that each tuple in Lemma \ref{lem.4cycle} (2) supports a twisted graded Calabi--Yau algebra with appropriate Nakayama automorphism.

\begin{theorem}\label{thm.4cycle}
Let $P$ be as in \eqref{eq.4cycle} and let $(M,P,\degs)$ be a type of a twisted graded Calabi--Yau algebra satisfying Hypothesis \ref{hyp.main} with $|Q_0|=4$. Then $\degs=3$ and $M$ is one of 
\[ C(1, 1, 0, 1), C(0, 0, 0, 3), C(0, 1, 0, 2), C(0, 1, 2, 0), C(1, 1, 1, 0), C(2, 1, 0, 0), C(0, 1, 1, 1)^*.\]
Conversely, each is the adjacency matrix of a quiver that supports a twisted graded Calabi--Yau algebra whose Nakayama automorphism restricts to \eqref{eq.4cycle} on $Q_0$, except possibly the one marked by $^*$.
\end{theorem}
\begin{proof}
The matrix $C(1,1,0,1)$ was constructed in Example \ref{ex.ore} (1) using Ore extensions. We can realize most of the remaining matrices using McKay quivers. Let $R=\kk[x,y,z]$. In the table below we give the action of $C_4=\grp{g}$ on $R$ along with the corresponding McKay quiver:
\[ \begin{array}{|c|c|}\hline
g & M \\ \hline
\diag(i,i,i) 		& C(0,0,0,3) \\
\diag(-1,-1,-i) 	& C(0,1,2,0) \\
\diag(i,i,-i)		& C(0,0,2,1) \\
\diag(1,-1,i) 	& C(1,0,1,1) \\
\diag(1,1,i) 	& C(2,1,0,0) \\
\hline
\end{array}\]
where $i=\sqrt{-1}$. In each case, the skew group ring $R\#\kk C_4$ is isomorphic to a derivation-quotient algebra on a quiver with the corresponding adjacency matrix.
\end{proof}

Given the natural symmetry of the quiver, it is surprising that $C(0, 1, 1, 1)$ does not appear as the McKay quiver for an action on $\kk[x,y,z]$ with a four-cycle permutation matrix. In Lemma \ref{lem.22iff} below, we show that this quiver appears if and only if a certain quiver in Theorem \ref{thm.22cycle} appears.

\subsection{Three cycle}\label{sec.3cycle}

In this case we consider when $P$ corresponds to a three-cycle. Explicitly, we take
\begin{align}\label{eq.3cycle}
P = \begin{psmatrix}0 & 0 & 1 & 0 \\ 1 & 0 & 0 & 0 \\ 0 & 1 & 0 & 0 \\ 0 & 0 & 0 & 1\end{psmatrix}.
\end{align}

\begin{lemma}\label{lem.3cycle1}
Let $P$ be as in \eqref{eq.3cycle} and let $(M,P,d)$ be a type of an algebra satisfying Hypothesis \ref{hyp.main} with $|Q_0|=4$. Then the matrix $M$ has the form
\begin{align}\label{eq.3cycleM}
M = \begin{psmatrix}w & x & y & u\\y & w & x & u \\x & y & w & u \\ u & u & u & r\end{psmatrix}
\end{align}
where the entries are nonnegative integers satisfying $w \leq s+1$ and $r \in \{0,1,2,3\}$ with $3w+r \leq 4\degs$. In particular, $M$ is normal.
\end{lemma}
\begin{proof}
In order for a matrix $M$ to commute with $P$, $M$ must be of the form
\[ M = \begin{psmatrix}w & x & y & v\\y & w & x & v \\x & y & w & v \\ u & u & u & r\end{psmatrix}.\]
Let $p(t)$ be as in \eqref{eq.mpoly} and set $D(t)=\det(p(t))$. 
Then $D(t)$ is palindromic by Lemma \ref{lem.palindrome}. By Theorem \ref{thm.mpoly}, $(t-1)^3$ divides $D(t)$. But then $D(t)/(t-1)^3$ is antipalindromic and so has a root of $1$.
Moreover, $D(t)$ is palindromic and a product of cyclotomic polynomials. Hence we have
\begin{align}\label{eq.3vertD}
D(t)= (1-t)^4 \prod_{i=1}^{2(\degs-1)} (1-k_i t + t^2) \quad\text{for some } k_i \in \mathbb{R}, |k_i| \leq 2.
\end{align}
Comparing the $t$ coefficient of $D(t)$ with the generic form given in \eqref{eq.3vertD}, we have
\[ 3w+r = 4 + (k_1+\cdots+k_{2(\degs-1)}) \leq 4\degs.\]
Hence, $w$ is at most $4$ if $\degs=3$ and $5$ if $\degs=4$. Because vertex 4 is fixed by $P$, then \cite[Proposition 3.8]{GR1} implies that $r$ is at most $6-\degs$.

We claim now that $u=v$, which will imply that $M$ is normal. By Theorem \ref{thm.mpoly}, $D(t)$ vanishes at $t=1$ to order at least $3$. Thus\footnote{Run mpoly(M,P,s) to obtain $D(t)$},
\[ 0 = D(1) = 9(u-v)^2(w-y-1)^2.\]
Hence, either $u=v$ or $w=y+1$.

If $u=v$, then we are done. Assume $w=y+1$. Making this substitution, we find that
\[ 0=D''(1) = 
\begin{cases}
6(u - v)^2(x - y + 1)^2 & \degs=3 \\
6(u - v)^2(2x - 2y + 1)^2 & \degs=4.
\end{cases}\]
If $\degs=4$, then we have $2x-2y=-1$, which is impossible since $x$ and $y$ are integers. We conclude that $\degs=3$ so we have $x=y-1$. Since $y+1=w\leq 4$, then making a choice from
\[\{(r,y) : r \in \{0,1,2,3\}, y \in \{0,1,2,3\}\}\]
determines a matrix whose remaining indeterminates are $u$ and $v$. From this matrix, we can compare to coefficients of $t^{10}$ in $D(t)$ compared to an arbitrary product of cyclotomic polynomials\footnote{Run y\_check().}. The only cases which have nonnegative integer solutions are those of the form $uv=0$ and $uv=1$. If $uv=0$, then the matrix is not strongly connected. If $uv=1$, then because the $u$ and $v$ are nonnegative integers, we conclude that $u=v$, as claimed.
\end{proof}

\begin{lemma}\label{lem.3cycle2}
Let $P$ be as in \eqref{eq.3cycle} and let $(M,P,\degs)$ be a type of an algebra satisfying Hypothesis \ref{hyp.main} with $|Q_0|=4$. Then $\degs=3$ and the possible matrices $M$ satisfying such that the corresponding $\det(p(t))$ \eqref{eq.mpoly} has all roots on the unit circle are as follows:
\[
\begin{psmatrix}0 & 0 & 0 & 1 \\0 & 0 & 0 & 1 \\ 0 & 0 & 0 & 1 \\ 1 & 1 & 1 & 2\end{psmatrix},
\begin{psmatrix}1 & 0 & 1 & 1 \\ 1 & 1 & 0 & 1 \\ 0 & 1 & 1 & 1 \\ 1 & 1 & 1 & 0\end{psmatrix},
\begin{psmatrix}0 & 1 & 1 & 1 \\ 1 & 0 & 1 & 1 \\ 1 & 1 & 0 & 1 \\ 1 & 1 & 1 & 0\end{psmatrix}.
%\begin{psmatrix}%1 & 1 & 0 & 1 \\ 0 & 1 & 1 & 1 \\ 1 & 0 & 1 & 1 \\ 1 & 1 & 1 & 0 \end{psmatrix},
\]
\end{lemma}
\begin{proof}
By Lemma \ref{lem.3cycle1}, we may assume that $M$ has the form \eqref{eq.3cycleM}, so $M$ is normal. By \cite[Proposition 8.10(3)]{RR1}, the spectral radius of $M$ is $6-\degs$. Let $\beta=1+\sqrt{-3}$. Then one finds that the eigenvalues of $M$ are
\begin{align*}
a &= -\frac{1}{2} \left( y\beta - x\overline{\beta}  + w\right), \qquad
b = \frac{1}{2} \left( y\overline{\beta} - x\beta  + w\right), \\
e_{\pm} &= \frac{1}{2} \left[ (w+x+y+r) \pm \sqrt{ ( w+x+y-r)^2 + 12u^2 } \right].
\end{align*}
Since $\rho(M)=6-s$, then we have $e_+ \leq 6-s$. Consequently, $w+x+y+r \leq 2e_+ \leq 2(6-\degs)$. By similar reasoning, $\sqrt{12u^2} \leq 2(6-\degs)$ so we obtain $u\in \{0,1\}$. If $u=0$ then the matrix is not strongly connected, so we conclude that $u=1$.

This leaves us with a finite number of matrices to check, which we do by computer\footnote{Run tre(s).}. We substitute in each possible set of values for the entries of $M$, then check the spectral radius, strong connectedness, and finally whether the resulting matrix polynomial appears as a product of cyclotomic polynomials. The result shows that no such matrices exist when $\degs=4$.
\end{proof}

We complete our analysis below.

\begin{theorem}\label{thm.3cycle}
Let $P$ be as in \eqref{eq.3cycle} and let $(M,P,\degs)$ be a type of an algebra satisfying Hypothesis \ref{hyp.main} with $|Q_0|=4$. Then $\degs=3$ and $M$ is one of the following:
\[
\begin{psmatrix}1 & 0 & 1 & 1 \\ 1 & 1 & 0 & 1 \\ 0 & 1 & 1 & 1 \\ 1 & 1 & 1 & 0\end{psmatrix},
\begin{psmatrix}0 & 0 & 0 & 1 \\ 0 & 0 & 0 & 1 \\ 0 & 0 & 0 & 1 \\ 1 & 1 & 1 & 2\end{psmatrix},
\begin{psmatrix}0 & 1 & 1 & 1 \\ 1 & 0 & 1 & 1 \\ 1 & 1 & 0 & 1 \\ 1 & 1 & 1 & 0\end{psmatrix}^*.
\]
Conversely, each is the adjacency matrix of a quiver that supports a twisted graded Calabi--Yau algebra whose Nakayama automorphism restricts to \eqref{eq.3cycle} on $Q_0$, except possibly the one marked by $^*$.
\end{theorem}
\begin{proof}
By Lemma \ref{lem.3cycle2}, it suffices to produce an algebra in each case with the corresponding adjacency matrix. The second case was outlined in Example \ref{ex.twist}.

%This is the first case
%For the first case, consider the McKay quiver of the $C_2 \times C_2 = \grp{g,h}$ action on $R=\kk[x,y,z]$ corresponding to the representation
%\[ g \to \diag(-1,1,-1), \quad h \to \diag(1,-1,-1).\]
%Then $A=R\#\kk(C_2 \times C_2)$ is isomorphic to a derivation-quotient algebra in which the adjacency matrix $M$ is the last matrix in the list above. However, it is clear that the permutation matrix $P$ in this case is the identity. 

For the first case, let $A=R\#\kk(C_2 \times C_2)$ be the algebra constructed in Example \ref{ex.mckay2}. Recall that $A$ has type $(M,I,3)$ where $M$ is the last matrix in the list above.
Let $\tau$ be the automorphism of $C_2 \times C_2$ given by $g \mapsto h \mapsto gh \mapsto g$. This extends to an automorphism of the skew group ring $A$ by the map $x \mapsto z \mapsto y \mapsto x$. The permutation matrix corresponding to $\tau$ is
\[ N = \begin{psmatrix}1 & 0 & 0 & 0 \\ 0 & 0 & 1 & 0 \\ 0 & 0 & 0 & 1 \\ 0 & 1 & 0 & 0\end{psmatrix}.\]
Then the twist ${}_\tau A$ is twisted graded Calabi--Yau of type $(NM,NP,3)$. It is easy to check that, up to reordering vertices, $NP$ is \eqref{eq.3cycle} and $NM$ is the first matrix in the statement of the theorem.
\end{proof}

\subsection{Two two-cycles}\label{sec.22cycle}

In this case, $P$ is the matrix representing two two-cycles. Explicitly, we take
\begin{align}\label{eq.22cycle}
P = \begin{psmatrix}0 & 1 & 0 & 0 \\ 1 & 0 & 0 & 0 \\ 0 & 0 & 0 & 1 \\ 0 & 0 & 1 & 0\end{psmatrix}.
\end{align}

\begin{lemma}\label{lem.22cycle}
Let $P$ be as in \eqref{eq.22cycle} and let $(M,P,\degs)$ be an algebra satisfying Hypothesis \ref{hyp.main} with $|Q_0|=4$. The matrix $M$ has the form 
\[ M =  \begin{psmatrix}a & b & c & d\\b & a & d & c \\e & f & g & h \\f & e & h & g\end{psmatrix}\]
with $a+g \leq 2\degs$. Without loss of generality, assume $g \leq a$ and set
\[ \gamma = (b-1)^2 + (h-1)^2 + 2ce + 2df - 2.\]
The following tables list all valid possibilities for $(a,g)$ along with the corresponding maximum values of $\gamma$:
\[
%s=3
\begin{array}{|c|c|c|c|c|}
\multicolumn{5}{c}{$\degs=3$} \\ \hline
(a,g) & \gamma_{\text{max}} & & (a,g) & \gamma_{\text{max}} \\ \hline
(0,0) & 6 & & (2,0) & 2 \\ 
(1,0) & 3 & & (2,1) & -1 \\
(1,1) & 4 & & (2,2) & -2 \\ \hline
\multicolumn{5}{c}{~}
\end{array}\qquad
%s=4
\begin{array}{|c|c|c|c|c|}
\multicolumn{5}{c}{$\degs=4$} \\ \hline
(a,g) & \gamma_{\text{max}} & & (a,g) & \gamma_{\text{max}} \\ \hline
(0,0) & 8 & & (2,0) & 4 \\ 
(1,0) & 5 & & (2,1) & 1 \\ 
(1,1) & 6 & & (2,2) & 0 \\
	&   & & (3,1) & -2 \\\hline
\end{array}
\]
\end{lemma}
\begin{proof}
The matrix $M$ has the given form since $M$ commutes with $P$. Let $p(t)$ be as in \eqref{eq.mpoly} and set $D(t)=\det(p(t))$.
Using the same argument as in Lemma \ref{lem.3cycle1} implies that $D(t)$ has the form given in \eqref{eq.3vertD}. Comparing the $t$ coefficient of the two forms gives
\[ 2(a+g) = 4 + (k_1+\cdots+k_{2(\degs-1)}) \leq 4\degs.\]
This implies that $a+g \leq 2\degs$. Define
\[ \beta = \sum_{i<j} m_{ii} m_{jj} = a^2+g^2+4ag.\]
Since the entries of $M$ are nonnegative integers, then it is clear that $\beta \geq 0$ and $\gamma \geq -2$. 

Fix a particular choice of $(a,g)$ with $a+g \leq 2\degs$, which determines $\beta$. Since the $t^2$ coefficient of $D(t)$ is $\beta-\gamma$, then we can compare this to the $t^2$ coefficient in all possible products of cyclotomic polynomials\footnote{Run tt\_gamcheck(s).}. The tables list all possible tuples in which the maximum possible value of $\gamma$ is at least $-2$.
\end{proof}

The remainder of the argument to produce possible types relies heavily on a computer computation in Sage, with some help from Mathematica. We describe our algorithm here. This is approximately the same as the procedure in \cite[Section 6]{GR1}, but we have made several improvements and streamlined the code\footnote{This is implemented entirely using total\_list(s).}.
\begin{itemize}
\item Fix $\degs$.
\item Choose a pair of diagonal entries $(a,g)$ from the table. 
\item Choose $\gamma$ such that $\gamma \leq \gamma_{\text{max}}$. In the code we use $z=\gamma+2$ so that all values are positive.
\item Choose a generic form $M$ corresponding to the $\gamma$ value. Note that there may be many such forms\footnote{These are created using tbuilder(z).}. In this case, there are at most two indeterminates for any fixed value of $\gamma$.
\item Choose a product of cyclotomic polynomials\footnote{This is lam\_polys($\lambda,P,s$) where $\lambda =-2(a + g)$.} $q(t)$ whose $t^{11}$ coefficient is equal to $-2(a+g)$.
\item Compute the matrix polynomial $p(t)$ of $M$ and compare $\det(p(t)) = q(t)$. We form a system of equations from the first nonzero coefficient of $q(t)-\det(p(t))$ and $\det(p(t))$ evaluated at $1$. We add to this any inequalities that might be forced from the variables so that the matrix is strongly connected. This is then fed into Mathematica to give integer solutions\footnote{All of this is implemented through zsolver()}.
\end{itemize}
This algorithm produces the following list of matrices
\begin{align*}
&\degs=3: &
%C_4 action
&\begin{psmatrix}1 & 0 & 2 & 0 \\ 0 & 1 & 0 & 2 \\ 0 & 2 & 1 & 0 \\ 2 & 0 & 0 & 1\end{psmatrix},
\begin{psmatrix}0 & 1 & 1 & 1 \\ 1 & 0 & 1 & 1 \\ 1 & 1 & 0 & 1 \\ 1 & 1 & 1 & 0\end{psmatrix},
%C_2 \times C_2 action
\begin{psmatrix}1 & 0 & 1 & 1 \\ 0 & 1 & 1 & 1 \\ 1 & 1 & 1 & 0 \\ 1 & 1 & 0 & 1\end{psmatrix},
\begin{psmatrix}0 & 1 & 2 & 0 \\ 1 & 0 & 0 & 2 \\ 2 & 0 & 0 & 1 \\ 0 & 2 & 1 & 0\end{psmatrix},
%Good ones, maybe
\begin{psmatrix}1 & 1 & 1 & 0 \\ 1 & 1 & 0 & 1 \\ 0 & 1 & 1 & 1 \\ 1 & 0 & 1 & 1\end{psmatrix},
\begin{psmatrix}1 & 1 & 1 & 0 \\ 1 & 1 & 0 & 1 \\ 1 & 0 & 0 & 2 \\ 0 & 1 & 2 & 0\end{psmatrix},
\begin{psmatrix}1 & 0 & 1 & 1 \\ 0 & 1 & 1 & 1 \\ 1 & 1 & 0 & 1 \\ 1 & 1 & 1 & 0\end{psmatrix} \\
%normal, spectral radius not 3
& & &\begin{psmatrix}1 & 0 & 1 & 0 \\ 0 & 1 & 0 & 1 \\ 0 & 1 & 1 & 0 \\ 1 & 0 & 0 & 1\end{psmatrix},
\begin{psmatrix}1 & 1 & 1 & 0 \\ 1 & 1 & 0 & 1 \\ 1 & 0 & 0 & 0 \\ 0 & 1 & 0 & 0\end{psmatrix},
\begin{psmatrix}0 & 0 & 1 & 0 \\ 0 & 0 & 0 & 1 \\ 1 & 0 & 0 & 0 \\ 0 & 1 & 0 & 0\end{psmatrix},
%bad Hilbert series
\begin{psmatrix}1 & 1 & 1 & 0 \\ 1 & 1 & 0 & 1 \\ 0 & 1 & 1 & 0 \\ 1 & 0 & 0 & 1\end{psmatrix},
\begin{psmatrix}1 & 0 & 1 & 0 \\ 0 & 1 & 0 & 1 \\ 0 & 1 & 0 & 1 \\ 1 & 0 & 1 & 0\end{psmatrix},
\begin{psmatrix}1 & 0 & 1 & 0 \\ 0 & 1 & 0 & 1 \\ 0 & 1 & 0 & 0 \\ 1 & 0 & 0 & 0\end{psmatrix}, \\
&\degs=4: &
&\begin{psmatrix}1 & 0 & 1 & 0 \\ 0 & 1 & 0 & 1 \\ 0 & 1 & 1 & 0 \\ 1 & 0 & 0 & 1\end{psmatrix},
\begin{psmatrix}0 & 0 & 1 & 1 \\ 0 & 0 & 1 & 1 \\ 1 & 1 & 0 & 0 \\ 1 & 1 & 0 & 0\end{psmatrix},
\begin{psmatrix}1 & 0 & 1 & 0 \\ 0 & 1 & 0 & 1 \\ 0 & 1 & 0 & 1 \\ 1 & 0 & 1 & 0\end{psmatrix},
\begin{psmatrix}0 & 0 & 2 & 0 \\ 0 & 0 & 0 & 2 \\ 1 & 1 & 0 & 0 \\ 1 & 1 & 0 & 0\end{psmatrix},
%bad Hilbert series
\begin{psmatrix}1 & 0 & 1 & 0 \\ 0 & 1 & 0 & 1 \\ 0 & 1 & 0 & 0 \\ 1 & 0 & 0 & 0\end{psmatrix},
\begin{psmatrix}0 & 1 & 1 & 0 \\ 1 & 0 & 0 & 1 \\ 0 & 1 & 0 & 1 \\ 1 & 0 & 1 & 0\end{psmatrix},
%normal, spectral radius not 2
\begin{psmatrix}0 & 0 & 1 & 0 \\ 0 & 0 & 0 & 1 \\ 1 & 0 & 0 & 0 \\ 0 & 1 & 0 & 0\end{psmatrix}.
\end{align*}

We now rule out many of the above matrices.

\begin{lemma}\label{lem.22ruleout}
\begin{enumerate}
\item The following matrices are normal but do not have spectral radius $6-\degs$:
\[ \degs=3: \begin{psmatrix}1 & 0 & 1 & 0 \\ 0 & 1 & 0 & 1 \\ 0 & 1 & 1 & 0 \\ 1 & 0 & 0 & 1\end{psmatrix},
\begin{psmatrix}1 & 1 & 1 & 0 \\ 1 & 1 & 0 & 1 \\ 1 & 0 & 0 & 0 \\ 0 & 1 & 0 & 0\end{psmatrix},
\begin{psmatrix}0 & 0 & 1 & 0 \\ 0 & 0 & 0 & 1 \\ 1 & 0 & 0 & 0 \\ 0 & 1 & 0 & 0\end{psmatrix}, \qquad
\degs=4: \begin{psmatrix}0 & 0 & 1 & 0 \\ 0 & 0 & 0 & 1 \\ 1 & 0 & 0 & 0 \\ 0 & 1 & 0 & 0\end{psmatrix}.\]

\item The following matrices produce Hilbert series with negative entries:
\[ \degs=3: \begin{psmatrix}1 & 1 & 1 & 0 \\ 1 & 1 & 0 & 1 \\ 0 & 1 & 1 & 0 \\ 1 & 0 & 0 & 1\end{psmatrix},
\begin{psmatrix}1 & 0 & 1 & 0 \\ 0 & 1 & 0 & 1 \\ 0 & 1 & 0 & 1 \\ 1 & 0 & 1 & 0\end{psmatrix},
\begin{psmatrix}1 & 0 & 1 & 0 \\ 0 & 1 & 0 & 1 \\ 0 & 1 & 0 & 0 \\ 1 & 0 & 0 & 0\end{psmatrix}, \qquad
\degs=4: \begin{psmatrix}1 & 0 & 1 & 0 \\ 0 & 1 & 0 & 1 \\ 0 & 1 & 0 & 0 \\ 1 & 0 & 0 & 0\end{psmatrix},
\begin{psmatrix}0 & 1 & 1 & 0 \\ 1 & 0 & 0 & 1 \\ 0 & 1 & 0 & 1 \\ 1 & 0 & 1 & 0\end{psmatrix}.\]
\end{enumerate}
\noindent Hence, for $P$ as in \eqref{eq.22cycle}, these matrices do not appear in any type of an algebra satisfying Hypothesis \ref{hyp.main} with $|Q_0|=4$.
\end{lemma}
\begin{proof}
Part (1) is easy to check. For part (2), the Hilbert series can be worked out inductively using the expression of $p(t)$ in Theorem \ref{thm.mpoly}. Alternatively, this can be computed using the code in Sage\footnote{Run hilb\_series$(M,P,n,\degs)$.}. In this way, for example, we find that the first matrix for $\degs=3$ has negative entries in degree $4$.
\end{proof}

Before proceeding to the main theorem, we demonstrate how to produce one of the types above, as the explanation is a bit lengthier than some others. We also show that one additional type occurs if and only if one of our missing types from the four-cycle case appears.

\begin{lemma}\label{lem.22bigG}
Let $P$ be as in \eqref{eq.22cycle} and let
\[ M = \begin{psmatrix}1 & 1 & 1 & 0 \\ 1 & 1 & 0 & 1 \\ 1 & 0 & 0 & 2 \\ 0 & 1 & 2 & 0\end{psmatrix}.\]
Then $(M,P,3)$ is the type of an algebra satisfying Hypothesis \ref{hyp.main} with $|Q_0|=4$.
\end{lemma}
\begin{proof}
This is a modification of \cite[Example 2.2]{GR1}.
Consider the group below which appears in \cite{CHI},
\[ G = \langle \sigma,\tau,\lambda \mid \sigma^8 = \lambda^8 = \tau^2, \lambda\sigma=\sigma\lambda, \lambda\tau=\tau\lambda, \lambda\sigma\tau=\tau\sigma\inv\rangle.\]
Let $\tornado$ be a primitive $8$th root of 1. Assume that $V$ is the 2-dimensional representation of $G$ given by
\[
\sigma \mapsto \begin{psmatrix}\tornado & 0 \\ 0 & \tornado\inv \end{psmatrix}, \quad
\tau \mapsto \begin{psmatrix}0 & 1 \\ 1 & 0 \end{psmatrix}, \quad
\lambda \mapsto \begin{psmatrix}1 & 0 \\ 0 & 1\end{psmatrix}.
\]
For each $i \in \{0,1,\hdots,7\}$ there is a 2-dimensional representation of $G$, denoted $W_i$, given by
\[
\sigma \mapsto \begin{psmatrix}\tornado^i & 0 \\ 0 & \tornado^{-i-1} \end{psmatrix}, \quad
\tau \mapsto \begin{psmatrix}0 & 1 \\ 1 & 0 \end{psmatrix}, \quad
\lambda \mapsto \begin{psmatrix}\tornado & 0 \\ 0 & \tornado\end{psmatrix}.
\]
It is clear, via the coordinate switch, that $W_i \iso W_{7-i}$ as representations. However, the representations $W_0,W_1,W_2,W_3$ are seen to be non-isomorphic by comparing eigenvalues of $\sigma$.

Note that $\lambda$ is central in $G$ of order $8$. Hence, any finite-dimensional representation $X$ of $G$ decomposes as $X = X_0 \oplus \cdots \oplus X_7$ where $X_i$ is a representation of $G$ on which $\lambda$ acts by scalar multiplication by $\tornado^i$. The action of $\lambda$ on $V$ is trivial. Consequently, the McKay quiver of the action decomposes as a disjoint union of quivers $Q_i$. In each $Q_i$, the vertices correspond to irreducible representations where $\lambda$ acts by $\tornado^i$.

One checks that $W_0,W_1,W_2,W_3$ are the only irreducible representations of $G$ in which $\lambda$ acts by $\tornado$. Furthermore, it is easy to verify that $V \tensor W_i \iso W_{i-1} \oplus W_{i+1}$ (where indices are taken mod $8$). Using the isomorphisms amongst the $W_i$ above, the component $Q_1$ of $Q$ is
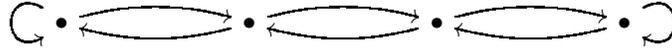
\[ 
  \xymatrix{
{}\bullet \ar@/^/[rr] \ar@(ul,dl)[] & &
{}\bullet \ar@/^/[rr] \ar@/^/[ll] & &
{}\bullet \ar@/^/[rr] \ar@/^/[ll] & &
\bullet \ar@(ur,dr)[] \ar@/^/[ll]}
\]
where the McKay matrix is
\[M' = \begin{psmatrix}1 & 1 & 0 & 0 \\ 1 & 0 & 1 & 0 \\ 0 & 1 & 0 & 1\\ 0 & 0 & 1 & 1\end{psmatrix}.\]

The group $G$ acts on $R=\kk_{-1}[x,y]$ by graded automorphism. The superpotential corresponding to $R$ is $\omega=xy+yx$ and it is clear that the $G$-action corresponding to $V$ is trivial on $\omega$. Then $R\#\kk G$ is Morita equivalent to a derivation-quotient algebra $A$ of type $(M',I,3)$.

Now observe that there is an automorphism $\rho$ of $\kk G$ given by $\sigma \mapsto \sigma\inv$, $\tau \mapsto \tau$, $\lambda \mapsto \lambda$. This automorphism interchanges $W_0 \leftrightarrow W_3$ and $W_1 \mapsto W_2$. Thus the matrix corresponding to this automorphism is 
\[ P' = \begin{psmatrix}0 & 0 & 0 & 1 \\ 0 & 0 & 1 & 0 \\ 0 & 1 & 0 & 0 \\ 1 & 0 & 0 & 0\end{psmatrix}.\]
Hence, the Ore extension $A[t;\rho]$ is twisted graded Calabi--Yau of dimension 3 of type $(M'+P',P',3)$.
%where
%\[ M' + P' = \begin{psmatrix}1 & 1 & 0 & 1 \\ 1 & 0 & 2 & 0 \\ 0 & 2 & 0 & 1\\ 1 & 0 & 1 & 1\end{psmatrix}.\]
A simple renumbering of vertices gives an algebra of type $(M,P,3)$ as in the statement.
%\[ \begin{psmatrix}1 & 1 & 1 & 0 \\ 1 & 1 & 0 & 1 \\ 1 & 0 & 0 & 2 \\ 0 & 1 & 2 & 0\end{psmatrix}\]
%with $P$ as given.
\end{proof}

\begin{lemma}\label{lem.22iff}
Let $P$ be as in \eqref{eq.22cycle} and let $P'$ be as in \eqref{eq.4cycle}. Let
\[ M = \begin{psmatrix}1 & 1 & 1 & 0 \\ 1 & 1 & 0 & 1 \\ 0 & 1 & 1 & 1 \\ 1 & 0 & 1 & 1\end{psmatrix}
\quad\text{and}\quad
M' = \begin{psmatrix}0 & 1 & 1 & 1 \\ 1 & 0 & 1 & 1 \\ 1 & 1 & 0 & 1 \\ 1 & 1 & 1 & 0\end{psmatrix}.\]
Then $(M,P,3)$ is the type of an algebra satisfying Hypothesis \ref{hyp.main} with $|Q_0|=4$ if and only if $(M',P',3)$ is.
\end{lemma}
\begin{proof}
Suppose $(M',P',3)$ is the type of a twisted graded Calabi--Yau algebra $A'$ satisfying Hypothesis \ref{hyp.main} with $|Q_0|=4$. Then $M'P'=P'M'$, and so we can consider the graded automorphism $\tau$ of $A'$ which cycles the vertices $1 \to 4 \to 3 \to 2 \to 1$. Thus, $P'$ is the automorphism of $\tau$ restricted to $(A')_0$. Then the graded twist of $A'$ by $\tau$ has type $NM'$ and $NP'$. Relabeling the vertices $2 \to 3 \to 4 \to 2$ gives the intended type $(M,P,3)$. The converse is similar.
\end{proof}

We are now ready for our main result of this section. 

\begin{theorem}\label{thm.22cycle}
Let $P$ be as in \eqref{eq.22cycle} and let $(M,P,\degs)$ be a type of an algebra satisfying Hypothesis \ref{hyp.main} with $|Q_0|=4$. Then $M$ is one of
\begin{align*}
&\degs=3: &
%C_4 action
&\begin{psmatrix}1 & 0 & 2 & 0 \\ 0 & 1 & 0 & 2 \\ 0 & 2 & 1 & 0 \\ 2 & 0 & 0 & 1\end{psmatrix},
\begin{psmatrix}0 & 1 & 1 & 1 \\ 1 & 0 & 1 & 1 \\ 1 & 1 & 0 & 1 \\ 1 & 1 & 1 & 0\end{psmatrix},
%C_2\times C_2 action
\begin{psmatrix}1 & 0 & 1 & 1 \\ 0 & 1 & 1 & 1 \\ 1 & 1 & 1 & 0 \\ 1 & 1 & 0 & 1\end{psmatrix},
\begin{psmatrix}0 & 1 & 2 & 0 \\ 1 & 0 & 0 & 2 \\ 2 & 0 & 0 & 1 \\ 0 & 2 & 1 & 0\end{psmatrix},
%Twist
\begin{psmatrix}1 & 0 & 1 & 1 \\ 0 & 1 & 1 & 1 \\ 1 & 1 & 0 & 1 \\ 1 & 1 & 1 & 0\end{psmatrix}
%Weird one
\begin{psmatrix}1 & 1 & 1 & 0 \\ 1 & 1 & 0 & 1 \\ 1 & 0 & 0 & 2 \\ 0 & 1 & 2 & 0\end{psmatrix},
%Maybe
\begin{psmatrix}1 & 1 & 1 & 0 \\ 1 & 1 & 0 & 1 \\ 0 & 1 & 1 & 1 \\ 1 & 0 & 1 & 1\end{psmatrix}^*, \\
&\degs=4: &
&\begin{psmatrix}1 & 0 & 1 & 0 \\ 0 & 1 & 0 & 1 \\ 0 & 1 & 1 & 0 \\ 1 & 0 & 0 & 1\end{psmatrix},
\begin{psmatrix}0 & 0 & 1 & 1 \\ 0 & 0 & 1 & 1 \\ 1 & 1 & 0 & 0 \\ 1 & 1 & 0 & 0\end{psmatrix},
\begin{psmatrix}1 & 0 & 1 & 0 \\ 0 & 1 & 0 & 1 \\ 0 & 1 & 0 & 1 \\ 1 & 0 & 1 & 0\end{psmatrix}^*,
\begin{psmatrix}0 & 0 & 2 & 0 \\ 0 & 0 & 0 & 2 \\ 1 & 1 & 0 & 0 \\ 1 & 1 & 0 & 0\end{psmatrix}^*.
\end{align*}
Conversely, each is the adjacency matrix of a quiver that supports a twisted graded Calabi--Yau algebra whose Nakayama automorphism restricts to \eqref{eq.4cycle} on $Q_0$, except possibly the ones marked by $^*$.
\end{theorem}
\begin{proof}
We begin with the case $\degs=3$. The first two arise as skew groups rings $\kk[x,y,z]\#\kk C_4$. Take $C_4=\grp{g}$, then the actions are given by $\diag(1,-i,-i)$ and $\diag(-1,i,-i)$, respectively, where $i=\sqrt{-1}$. The second two arise as skew groups rings $\kk[x,y,z]\#\kk(C_2 \times C_2)$. The action of $C_2\times C_2=\grp{g,h}$ is given by $(\diag(1,-1,1), \diag(1,1,-1))$ and $(\diag(-1,-1,1), \diag(1,1,-1))$, respectively.

Consider the skew group ring $A=R\#\kk(C_2 \times C_2)$ constructed in Example \ref{ex.mckay2}. Recall that $A$ has type $(M',I,3)$ where 
\[ M'= \begin{psmatrix}0 & 1 & 1 & 1 \\ 1 & 0 & 1 & 1 \\  1 & 1 & 0 & 1 \\ 1 & 1 & 1 & 0\end{psmatrix}.\]
%The action of $C_2 \times C_2=\grp{g,h}$ is given by $(\diag(-1,1,-1), \diag(1,-1,-1))$. This algebra is twisted graded Calabi--Yau with type $(M',I,3)$ where
%\[ M'= \begin{psmatrix}0 & 1 & 1 & 1 \\ 1 & 0 & 1 & 1 \\  1 & 1 & 0 & 1 \\ 1 & 1 & 1 & 0\end{psmatrix}.\]
Let $\tau$ be the automorphism of $\kk(C_2 \times C_2)$ determined by $g \mapsto -g$ and $h \mapsto h$. This extends to an automorphism of $A$ and the permutation matrix $P$ corresponding to $\tau$ is \eqref{eq.22cycle}. Hence, ${}_\tau A$ is twisted graded Calabi--Yau with type $(PM',P,3)$. 
%where \[ M=PM' = \begin{psmatrix}1 & 0 & 1 & 1 \\ 0 & 1 & 1 & 1 \\ 1 & 1 & 0 & 1 \\ 1 & 1 & 1 & 0\end{psmatrix}.\]
This gives the fifth matrix in the list.

The next-to-last matrix is detailed in Lemma \ref{lem.22bigG}. By Lemma \ref{lem.22iff}, the final matrix belongs to the type of an algebra satisfying Hypothesis \ref{hyp.main} if and only if the missing matrix in Theorem \ref{thm.4cycle} does with appropriate permutation matrices.

Now consider the case $\degs=4$. For the first matrix, consider the three-dimensional Artin--Schelter regular algebra $R$ with superpotential
\[ \omega = y^2x^2 + xy^2x + yx^2y + x^2y^2.\]
We let $C_4=\grp{g}$ act on $R$ by $g=\diag(1,i)$. Then $g(\omega)=-\omega$ and so $\hdet g = -1$. Thus, the McKay quiver of the action is the matrix given, and so $R\#\kk C_4$ is Morita equivalent to a derivation-quotient algebra on that quiver.

The argument for the second matrix is similar. Consider the three-dimensional Artin--Schelter regular algebra $R$ with superpotential
\[ \omega = x^2y^2 + ixy^2x - y^2x^2 -i yx^2y.\]
We let $C_4=\grp{g}$ act on $R$ by $g=\begin{psmatrix}0 & 1 \\ -1 & 0\end{psmatrix}$. Then $g(\omega)=-\omega$ and so $\hdet g = -1$.
\end{proof}

The two starred matrices in the $\degs=4$ case of Theorem \ref{thm.22cycle} are not normal. We expect that these do not appear as the adjacency matrix of an algebra satisfying Hypothesis \ref{hyp.main}, though we do not have a good argument for ruling them out.

\bibliographystyle{amsplain}
%\bibliography{biblio}{}

\end{document}